\documentclass[11pt, onecolumn]{article}
\usepackage[top=1in, bottom=1in, left=1.25in, right=1.25in]{geometry}
\usepackage{rawfonts,graphicx,amsfonts,amssymb,psfrag,flushend,amsmath}
\usepackage{amssymb}
\usepackage{amsmath}
\usepackage{bbding}
\usepackage{epsfig}
\usepackage{color, soul}
\usepackage{array}
\usepackage{multirow}
\usepackage{booktabs}
\usepackage{ifthen}
\usepackage{algorithm}
\usepackage{algorithmic}
\usepackage{subfigure}
\usepackage{amsthm}
\usepackage{bm}
\usepackage{url}
\usepackage{enumerate}
\usepackage{caption}
\newcommand{\argmin}{\arg\min\limits}
\newcommand{\innerprod}[2]{\left\langle{#1},{#2}\right\rangle}

\newtheorem{cor}{Corollary}
\newtheorem{lemma}{Lemma}
\newtheorem{defi}{Definition}
\newtheorem{rem}{Remark}
\newtheorem{thm}{Theorem}
\newtheorem{asp}{Assumption}

\begin{document}
\title{Linearized ADMM for Non-convex Non-smooth Optimization with Convergence Analysis}
\date{November 1, 2017}

\author{Qinghua~Liu, Xinyue~Shen, and Yuantao~Gu\thanks{Department of Electronic Engineering and Tsinghua National Laboratory for Information Science and Technology (TNList), Tsinghua University, Beijing 100084, CHINA. The corresponding author of this paper is Yuantao Gu (gyt@tsinghua.edu.cn).}}

\maketitle

\begin{abstract}
Linearized alternating direction method of multipliers (ADMM) as an extension of
ADMM has been widely used to solve linearly constrained problems in signal processing,
machine leaning, communications, and many other fields.
Despite its broad applications in nonconvex optimization,
for a great number of nonconvex and nonsmooth objective functions,
its theoretical convergence guarantee is still an open problem.
In this paper, we propose a two-block linearized ADMM
and a multi-block parallel linearized ADMM
for problems with nonconvex and nonsmooth objectives.
Mathematically, we present that the algorithms can converge
for a broader class of objective functions under less strict assumptions
compared with previous works.
Furthermore, our proposed algorithm can update coupled variables in parallel
and work for less restrictive nonconvex problems,
where the traditional ADMM may have difficulties
in solving subproblems.

\end{abstract}

\section{Introduction}

In signal processing \cite{SP}, machine learning \cite{ML},
and communication \cite{CO}, many of the recently most concerned problems,
such as compressed sensing \cite{compressed_sensing},
dictionary learning \cite{dictionary_learning}, and channel estimation \cite{channel},
can be cast as optimization problems.
In doing so, not only has the design of the solving methods been greatly facilitated,
but also a more mathematically understandable and manageable
description of the problems has been given.
While convex optimization has been well studied \cite{bertsekas1999nonlinear,bertsekas1989parallel,OP2},
nonconvex optimization has also appeared in numerous topics
such as nonnegative matrix factorization \cite{n_m_f}, phase retrieval \cite{p_r},
distributed matrix factorization \cite{d_m_f}, and distributed clustering \cite{d_c}.

The alternating direction method of multipliers (ADMM) is widely used in linearly constrained optimization problems arising in machine learning \cite{AP1,AP2},
signal processing \cite{Shen2}, as well as other fields \cite{AP3,AP4,AP5}.
First proposed in the early 1970s, it has been studied extensively \cite{ref_convex_1,ref_convex_2,ref_convex_3}. \
At the very beginning, ADMM was mainly applied in solving linearly constrained convex problems \cite{ADMM1} in the following form
\begin{equation}\label{prob:main:}
\begin{array}{ll}
\mbox{minimize} & f({\bf x})+h({\bf y}) \\
\mbox{subject to} & {\bf A}{\bf x}+{\bf B}{\bf y} = {\bf 0},
\end{array}
\end{equation}
where ${\bf x}\in\mathbb{R}^p, {\bf y}\in\mathbb{R}^q$ are variables,
and ${\bf A}\in \mathbb{R}^{n\times p},{\bf B}\in \mathbb{R}^{n\times q}$ are given.
With an augmented Lagrangian function defined as
 
\begin{align}\label{L1}
L_{\beta}({\bf x},{\bf y},{\bm\gamma})
= f({\bf x})+h({\bf y})+\langle{\bm\gamma},{\bf A}{\bf x}+{\bf B}{\bf y} \rangle +\frac{\beta}{2} \Vert {\bf A}{\bf x}+{\bf B}{\bf y} \Vert ^{2}_{2},
\end{align}
 
where ${\bm\gamma}$ is the Lagrangian dual variable,
the ADMM method updates variables iteratively as the following

\begin{align*}
&{\bf x}^{k+1}=\argmin_{{\bf x}} L_\beta ({\bf x},{\bf y}^{k},{\bm\gamma}^{k} ), \\
&{\bf y}^{k+1}=\argmin_{{\bf y}} L_\beta ({\bf x}^{k+1},{\bf y},{\bm\gamma}^{k} ), \\
&{\bm\gamma}^{k+1}={\bm\gamma}^{k}+\beta({\bf A}{\bf x}^{k+1}+{\bf B}{\bf y}^{k+1}).
\end{align*}

For ADMM applied in nonconvex problems,
although the theoretical convergence guarantee is still an open problem,
it can converge fast in many cases \cite{Shen,Chen}.
Under certain assumptions on the objective function and linear constraints,
researchers have studied the convergence of ADMM for nonconvex optimization \cite{Wang,Hong,Wang1,jiang2016structured,guo2017convergence,yang2017alternating,li2015global}.

The subproblems in ADMM can be hard to solve and have no closed form solution in many cases,
so we either use an approximate solution as a substitute in the update which might cause divergence,
or solve the subproblems by numerical algorithms which can bring computational burden.
Motivated by these issues, linearized ADMM was proposed for convex optimization \cite{lin2017extragradient,Linearized_1,Linearized_3,lin2011linearized,ling2015dlm,linear_sparse}. By linearizing the intractable part in subproblems, they make unsolvable problems solvable and reduce computational complexity. It has been applied in sparsity recovery \cite{lin2011linearized, linear_sparse,gu2014fast}, low-rank matrix completion \cite{linear_matrix}, and image restoration \cite{linear_image,nien2015fast,woo2013proximal,ng2011inexact}, and has demonstrated good performances.

When the problem scale is so large that
a two-block ADMM method may no longer be efficient or practical
\cite{bigdata_ml,bigdata_sp},
distributed algorithms are in demand to exploit
parallel computing resources \cite{bertsekas1989parallel,scutari2014distributed,scutari2016parallel}.
Multi-block ADMM was proposed to solve problems in the following form
\cite{multi-block1}
\begin{equation}\label{multi-block0}
\begin{array}{ll}
\mbox{minimize} & f_{1}({\bf x}_{1})+f_{2}({\bf x}_{2})+\cdots+f_{K}({\bf x}_{K}) \\
\mbox{subject to} & {\bf A}_{1}{\bf x}_{1}+{\bf A}_{2}{\bf x}_{2}+\cdots+{\bf A}_{K}{\bf x}_{K}= {\bf 0}.
\end{array}
\end{equation}
It allows parallel computation
\cite{ref_convex_1,Hong,he2016proximal,multi-parallel-1,multi-parallel-2,multi-parallel-3},
and has  been used in problems such as sparse statistical machine learning \cite{multi-block2}
and total variation regularized image reconstruction \cite{multi-block3}.

\subsection{Main problems}
In this paper, we study linearized ADMM algorithms
for problems with nonconvex and nonsmooth objective functions.
First, we propose a two-block linearized ADMM
for problems with \textit{coupled} variables in the following form
\begin{equation}\label{prob:main}
\begin{array}{ll}
\mbox{minimize} & g({\bf x},{\bf y})+f({\bf x})+h({\bf y}) \\
\mbox{subject to} & {\bf A}{\bf x}+{\bf B}{\bf y} = {\bf 0},
\end{array}
\end{equation}
where ${\bf x}\in\mathbb{R}^p,{\bf y}\in\mathbb{R}^q$ are variables. Functions $g$ and $h$ are differentiable and can be nonconvex. Function $f$ can be both nonconvex and nondifferentiable. The Lagrangian function for problem \eqref{prob:main} is defined as follows
 
\begin{align}
L_{\beta}({\bf x},{\bf y},{\bm\gamma}) =
 g({\bf x},{\bf y})+f({\bf x})+h({\bf y})
+\innerprod{{\bm\gamma}}{{\bf A}{\bf x}+{\bf B}{\bf y}} 
+\frac{\beta}{2}
\left\Vert {\bf A}{\bf x}+{\bf B}{\bf y} \right\Vert ^{2}_{2}.\label{Lagrangian}
\end{align}
 
Throughout, we make the following assumption.
\begin{asp}\label{asp1}
Assume that problem \eqref{prob:main} satisfies the conditions below.
\begin{itemize}
    \item[1.] Function $h({\bf y})$ is $L_h$-Lipschitz differentiable.
    \item[2.] Function $ g({\bf x},{\bf y})$ is $L_g$-Lipschitz differentiable.

    \item[3.] Function $ g({\bf x},{\bf y})+f({\bf x})+h({\bf y})$ is lower bounded and coercive  with respect to ${\bf y}$ over the feasible set
    \begin{displaymath}
    \left\{({\bf x},{\bf y})\in \mathbb{R}^{p+q}:{\bf A}{\bf x}+{\bf B}{\bf y}={\bf 0}\right\}.
    \end{displaymath}

    \item[4.] Matrix ${\bf B}$ has full column rank, and ${\bf Im}({\bf A})\subset {\bf Im}({\bf B})$.

\end{itemize}
\end{asp}
In Assumption \ref{asp1}, we put relatively weak restriction on function $f$ and matrix ${\bf A}$, which is a significant improvement over other nonconvex ADMM algorithms.

Then we propose a parallel multi-block ADMM method,
which can be seen as a special case of the first algorithm,
for problems in the following form
\begin{equation}\label{multi-block}
\begin{array}{ll}
\mbox{minimize} & g({\bf x}_{1},\ldots,{\bf x}_{K},{\bf y})+\sum_{i=1}^{K}f_{i}({\bf x}_{i})+h({\bf y}) \\
\mbox{subject to} & {\bf A}_{1}{\bf x}_{1}+\cdots+{\bf A}_{K}{\bf x}_{K}+{\bf B}{\bf y} = {\bf 0},
\end{array}
\end{equation}
where ${\bf x} =({\bf x}_{1},\ldots,{\bf x}_{K})$ and ${\bf y}$ are variables. The assumption we have on problem \eqref{multi-block} is the same as Assumption \ref{asp1}.
\subsection{Related Works}
Recently a great deal of attention has been focused on using ADMM to solve nonconvex problems \cite{Wang,Hong,Wang1,jiang2016structured,guo2017convergence,yang2017alternating,li2015global}.
The work \cite{Wang} studies the convergence of traditional ADMM under relatively strict assumptions. For instance, it requires every ${\bf A}_i$ to have full column rank and all the $f_i$ to satisfy an assumption similar to Holder condition. Besides, the parameter $\beta$ in their algorithm is required to increase linearly in the number of variable blocks, which can seriously reduce its convergence speed. The work \cite{Hong} studies the convergence of ADMM for solving nonconvex consensus and sharing problem. However, they require the nonconvex part to be Lipschitz differentiable and the nondifferentiable part to be convex. The work \cite{Hong} also studies a parallel ADMM, but it is only under the case where the Lagarangian function is separable for each block, that is, the objective function and augmented term are both separable. The work \cite{jiang2016structured} studies nonconvex ADMM under less restrictive assumptions. Their algorithm requires matrix ${\bf B}$ to have full row rank, while our algorithm requires matrix ${\bf B}$ to have full column rank, so their algorithm adapts to different optimization problems from ours. In addition, our second algorithm allows parallel computation for multi-block cases, while theirs does not.


Besides ADMM there are also other kinds of dual algorithms for multi-block nonconvex optimization. For instance, \cite{scutari2014distributed} studies a distributed dual algorithm for nonconvex constrained problem, where the integral objective function is Lipschitz differentiable and the Lagrangian function is defined without the augmented term. It can be viewed as a variation of the \emph{method of Lagrangian multiplier}, while our algorithms are variations of the \emph{Augmented Lagrangian method}. In addition, our algorithms can adapt to nonsmooth optimization even with indicator functions in the objective, while their algorithm can not.

\subsection{Contribution}

Our work has the following improvements
compared with some latest works based on ADMM
for nonconvex optimization.

\begin{itemize}
\item \textbf{Nonconvex linearized ADMM:} This is the first work to study theoretical convergence for
linearized ADMM in nonconvex optimization. By linearizing all the differentiable parts, not only the objective function but also the augmented term,  in the Lagrangian function, the subproblems
can either be transformed into a proximal problem or a quadratic problem, which are usually easier to solve than the original subproblems.
\item \textbf{Parallel Computation:} In our second algorithm, the linearization decouples the variables ${\bf x}_{1},\ldots,{\bf x}_{K}$ originally coupled in the function $g$ and $\frac{\beta}{2} \Vert \sum_{i=1}^{K}{\bf A}_{i}{\bf x}_{i}+{\bf B}{\bf y} \Vert ^{2}_{2}$, so we can update every block in parallel.
Previous works \cite{ref_convex_1,multi-parallel-1,multi-parallel-2,multi-parallel-3} have studied some parallel ADMM algorithms that can deal with coupled variables, but they are all for convex optimization. To the best of our knowledge, our second algorithm is the first one to extend such parallel ADMM to nonconvex optimization. Numerical experiment demonstrates the high efficiency of our algorithm brought by parallel computation in comparison with other latest nonconvex ADMM algorithms.

\item \textbf{Weaker assumptions:}  Our assumptions are less restrictive in comparison with previous works on nonconvex ADMM (see, e.g., \cite{Wang,Hong,Wang1,guo2017convergence,yang2017alternating,li2015global}). Specifically, we put much weaker restriction on function $f$ ($f_i$) and matrix ${\bf A}$ (${\bf A}_i$). The work \cite{jiang2016structured} has assumptions similar to ours, but the update rules are different,
and their algorithm requires matrix ${\bf B}$ to have full row rank, while we require matrix ${\bf B}$ to have full column rank.


\end{itemize}

\subsection{Outline}
The remainder of this paper is organized as follows.
In Section \ref{sec:pre} some preliminaries are introduced.
In Section \ref{two-block} we propose a two-block linearized ADMM for nonconvex problems
and  provide convergence analysis under certain broad assumption in Section \ref{two-blcok-analysis}.
In Section \ref{sec:multi-b-P-L-ADMM} we propose
a parallel muti-block linearized ADMM that can be seen as a special case of the first algorithm.
Section \ref{sec:discussion} gives detailed discussions on
the update rules and some applications to demonstrate the advantages of this work. In section \ref{sec:exp},
numerical experiments are performed to demonstrate the effectiveness and high efficiency of our algorithms.
We conclude this work in Section \ref{sec:conclusion}.

\section{Preliminary}
\label{sec:pre}

\subsection{Notation}
We use bold capital letters for matrices,
bold small case letters for vectors,
and non-bold letters for scalars.
We use ${\bf x}^{k}$ to denote the value of ${\bf x}$ after $k$th iteration
and ${\bf x}_{i}$ to denote its $i$th block.
The gradient of function $f$
at ${\bf x}$ for the $i$th component
is denoted as $\nabla_{{\bf x}_i} f({\bf x})$,
and the \textit{regular subgradient} of $f$ for the $i$th component
which is defined at a point ${\bf x}$ \cite{subgradient},
is denoted as $\partial_{i}f({\bf x})$.
The smallest eigenvalue of matrix ${\bf X}$ is denoted as $\lambda_{\bf X}$.
Without specification, $\|\cdot\|$ denotes $\ell_2$ norm. ${\bf Im}({\bf X})$ denotes the image of matrix ${\bf X}$.
In multi-block ADMM, ${\bf x} = \left[{\bf x}_1^{\rm T}, \dots, {\bf x}_K^{\rm T}\right]^{\rm T}$ denotes the collection of variables.

\subsection{Definition}

\begin{defi}\label{regular-subgradient}{\bfseries (Regular Subgradient)} \cite{subgradient}
Consider a function $f:\mathbb{R}^{n}\rightarrow \bar{\mathbb{R}}$ and  a point ${\bf x}_0$ with $f({\bf x}_0)$ finite. Then the regular subgradient of function $f$ at ${\bf x}_0$ is defined as
 \begin{align*}
\partial f({\bf x}_0) = \{{\bf v}:f({\bf x})\geq f({\bf x}_0) + \innerprod{{\bf v}}{{\bf x}-{\bf x}_0}+o(\Vert {\bf x}-{\bf x}_0\Vert)\},
\end{align*} 
where for every ${\bf v}$ the inequality holds for any $x$ in a small neighborhood of ${\bf x}_0$.
\end{defi}
  
\begin{rem}
Notice that the regular subgradient is a set.
For a differentiable function, its regular subgradient set at a point contains only its gradient at that point.
\end{rem}

\begin{defi} ({\bfseries Lipschitz Differentiable})
 Function $s({\bf y})$ is said to be  \\  $L_s$-Lipschitz differentiable if for all ${\bf y},{\bf y}^\prime$, we have
$$
\Vert\nabla s({\bf y})-\nabla s({\bf y}^\prime)\Vert_2\leq L_{s}\Vert {\bf y}-{\bf y}^\prime\Vert_2,
$$
equivalently, its gradient $\nabla s$ is Lipschitz continuous.
\end{defi}

\begin{defi}\label{coercive-func}({\bfseries Coercive Function})
Assume that function $r({\bf x}_1,{\bf x}_2)$ is defined on $\mathcal{X}$,
and for any $\Vert {\bf x}_2^k\Vert \rightarrow +\infty$ and $({\bf x}_1^k,{\bf x}_2^k) \in \mathcal{X}$,
we have $r({\bf x}_1^k,{\bf x}_2^k) \rightarrow +\infty$,
then function $r$ is said to be coercive with respect to ${\bf x}_2$ over $\mathcal{X}$.
\end{defi}
  
\begin{rem}
Any function is coercive over bounded set.
\end{rem}

%
%
%
%

\begin{algorithm}[t]
\caption{Two-block linearized ADMM algorithm}
\label{algorithm1}
\begin{algorithmic}
\STATE{ Initialize ${\bf x}^{0}, {\bf y}^{0}, {\bm\gamma}^{0}$.}
\WHILE{$\max\{\|{\bf x}^{k}-{\bf x}^{k-1}\|,\|{\bf y}^{k}-{\bf y}^{k-1}\|,\|{\bm\gamma}^{k}-{\bm\gamma}^{k-1}\|\}>\varepsilon $}
\STATE{${\bf x}^{k+1}=\argmin_{{\bf x}} \bar{f}^{k}({\bf x})$}
\STATE{${\bf y}^{k+1}=\argmin_{{\bf y}} \bar{h}^k({\bf y})$}
\STATE{${\bm\gamma}^{k+1}={\bm\gamma}^{k}+\beta({\bf A}{\bf x}^{k+1}+{\bf B}{\bf y}^{k+1})$}
\STATE{$k = k+1$}
\ENDWHILE
\RETURN $({\bf x}^{k}, {\bf y}^{k}, {\bm\gamma}^{k})$
\end{algorithmic}
\end{algorithm}

\section{Linearized ADMM: two-block and multi-block}
\label{sec:two-b-L-ADMM}

In this section, we first propose a linearized ADMM to solve
the two-block nonconvex  problem \eqref{prob:main}
possibly with function $f$ nonsmooth.
Its convergence assumption is, as far as we know,
one of the broadest among the current ADMM algorithms for nonconvex optimization.
Then we extend the algorithm to solve the multi-block problem \eqref{multi-block},
and the linearization renders the coupled multi-blocks of variables to be updated in parallel.

\subsection{Two-block linearized ADMM updating rules}\label{two-block}

In the $(k+1)$th update of ${\bf x}$, we replace $g({\bf x},{\bf y})+\frac{\beta}{2}\Vert {\bf A}{\bf x}+{\bf B}{\bf y}\Vert^{2}$ by its approximation
$$
\langle {\bf x}-{\bf x}^{k},\nabla_{{\bf x}} g({\bf x}^{k},{\bf y}^{k})+\beta {\bf A}^{\rm T}({\bf A}{\bf x}^{k}+{\bf B}{\bf y}^{k})\rangle + \frac{L_{x}}{2}\Vert {\bf x}-{\bf x}^{k}\Vert^{2},
$$
which is a linearized term plus a regularization term ($L_{x}>0$).
In the $(k+1)$th update of $\bf y$,
the algorithm replaces $g({\bf x},{\bf y})+h({\bf y})$ by its approximation
$$
\langle {\bf y}-{\bf y}^{k},\nabla_{{\bf y}} g({\bf x}^{k+1},{\bf y}^{k})+\nabla h({\bf y}^{k})\rangle + \frac{L_{y}}{2}\Vert {\bf y}-{\bf y}^{k}\Vert^{2},
$$
which is again a linearized term plus a regularization term ($L_{y}>0$). 
Replacing the corresponding parts in Lagrangian function with their approximations derived above, we readily get the following two auxiliary functions.
 \begin{align}
\bar{f}^{k}({\bf x})=&f({\bf x})+ \langle{\bm\gamma}^{k},{\bf A}{\bf x}\rangle+\frac{L_{x}}{2}\Vert {\bf x}-{\bf x}^{k}\Vert^{2}\nonumber \\& + \langle {\bf x}-{\bf x}^{k},\nabla_{{\bf x}} g({\bf x}^{k},{\bf y}^{k})+\beta {\bf A}^{\rm T}({\bf A}{\bf x}^{k}+{\bf B}{\bf y}^{k})\rangle , \label{eq-twoblock-x-update}\\
\bar{h}^k({\bf y})= &\langle{\bm\gamma}^{k},{\bf B}{\bf y}\rangle + \frac{L_{y}}{2}\Vert {\bf y}-{\bf y}^{k}\Vert^{2}+\frac{\beta}{2}\Vert {\bf A}{\bf x}^{k+1}+{\bf B}{\bf y}\Vert^{2}\nonumber \\&+\langle {\bf y}-{\bf y}^{k},\nabla_{{\bf y}} g({\bf x}^{k+1},{\bf y}^{k})+\nabla h({\bf y}^{k})\rangle.\label{eq-twoblock-y-update}
\end{align} 
Utilizing the two auxiliary functions above, the update rules are summarized in Algorithm \ref{algorithm1}.
Note that the ${\bf x}$ and ${\bf y}$ update rules in Algorithm \ref{algorithm1} can be simplified into the following form
		 \begin{align*}
		{\bf x}^{k+1}=&{\bf prox}_{f/L_x}\left\{ {\bf x}^{k}- \frac{1}{L_x}\left[\nabla_{{\bf x}} g({\bf x}^{k},{\bf y}^{k})+{\bf A}^{\rm T} {\bm\gamma}^{k}
		+\beta {\bf A}^{\rm T}({\bf A}{\bf x}^{k}+{\bf B}{\bf y}^{k})\right]\right\};\\
		{\bf y}^{k+1}=&\left(L_{y}+\beta {\bf B}^{\rm T} {\bf B}\right)^{-1}\big(L_{y}{\bf y}^{k}-\nabla_{{\bf y}} g({\bf x}^{k+1},{\bf y}^{k})-\nabla h({\bf y}^{k})-{\bf B}^{\rm T}{\bm\gamma}^{k}-\beta {\bf B}^{\rm T}{\bf A}{\bf x}^{k+1}\big).
		\end{align*} 
The subproblem in updating ${\bf x}$ is formulated into a proximal problem,
which can be easier to solve than the original subproblem and even have closed form solution \cite{parikh2014proximal}.
The matrix inversion in the ${\bf y}$-updating step can be computed beforehand,
so we do not need to compute it in every iteration.

\subsection{Convergence analysis}\label{two-blcok-analysis}

We give convergence analysis for Algorithm \ref{algorithm1} under Assumption \ref{asp1}.
Note that in this part, we refer $L_\beta$
to the augmented Lagrangian function
defined in \eqref{Lagrangian}.
To begin with, we show that $L_\beta$
and the primal and dual residues are able to converge in the following theorem.
  
\begin{thm}\label{convergence1}
For the linearized ADMM in Algorithm \ref{algorithm1}, under Assumption \ref{asp1}, if  we choose parameters $L_{x}$, $L_{y}$, and $\beta$ as follows
\begin{equation}\label{bb}
\begin{array}{l}
L_x \ge  L_g + \beta L_{\bf A} + 6L_w^2+1,\\
L_y \ge L_w +L_w^2 +3,\\
C_m = \frac{L_y+L_w^2}{2},\\
 \beta \ge \max\left\{\frac{L_w+L_y+2}{\lambda_{{\bf B}^{\rm T}{\bf B}}},\frac{3(L_w^2+L_y^2)}{\lambda_{{\bf B}^{\rm T}{\bf B}}C_m},\frac{3L_y^2}{\lambda_{{\bf B}^{\rm T}{\bf B}}}\right\},
\end{array}
\end{equation}
where $L_{{\bf A}}$ is the largest eigenvalue  of ${\bf A}^{\rm T}{\bf A}$, $\lambda_{{\bf B}^{\rm T}{\bf B}}$ is the smallest eigenvalue of ${\bf B}^{\rm T}{\bf B}$ and $L_w = L_g + L_h$,
then $\{L_{\beta}({\bf x}^{k},{\bf y}^{k},{\bm\gamma}^{k})\}$ is convergent,
and the primal residues $\Vert {\bf y}^{k+1}-{\bf y}^{k}\Vert$, $\Vert {\bf x}^{k+1}-{\bf x}^{k} \Vert$
and dual residue $\Vert {\bm\gamma}^{k+1}-{\bm\gamma}^{k}\Vert$ converge to zero as $k$ approaches infinity.
\end{thm}

\begin{proof}
We briefly introduce the structure of the proof here and the detailed version is  postponed to Appendix \ref{proof-convergence1}.
  
First, we will prove that the descent of $L_\beta$ after the $(k+1)$th iteration of ${\bf x}$ is lower bounded by $\Vert {\bf x}^{k+1}-{\bf x}^k\Vert$, the descent of $L_\beta$ after the $(k+1)$th iteration of ${\bf y}$ is lower bounded by $\Vert {\bf y}^{k+1}-{\bf y}^k\Vert$, and the ascent of $L_\beta$ after the $(k+1)$th iteration of ${\bf \gamma}$ is upper bounded by $\Vert {\bf x}^{k+1}-{\bf x}^k\Vert$, $\Vert {\bf y}^{k+1}-{\bf y}^k\Vert$ and $\Vert {\bf y}^{k}-{\bf y}^{k-1}\Vert$. Then, we will elaborately design an auxiliary sequence and prove its monotonicity and convergence. Finally, based on these conclusions, we will obtain the convergence of $L_\beta$ and both the primal and dual residues.
\end{proof}
  
Theorem \ref{convergence1} illustrates that
the function $L_\beta$ will converge,
and the increments of ${\bf x}$, ${\bf y}$, and ${\bm\gamma}$ after one iteration,
which are the primal and dual residues, will converge to zero.
  
\begin{cor}\label{bound1}
For the  linearized ADMM in Algorithm \ref{algorithm1},  under Assumption \ref{asp1}
together with function $g({\bf x},{\bf y})$ degenerating to $g({\bf x})$,
if we choose the parameters $L_x$, $L_{y}$ and $\beta$ satisfying \eqref{bb}, then
the generated dual variable sequence $\{{\bm\gamma}^{k}\}$ is bounded.
\end{cor}
  
\begin{proof}
The proof is postponed to Appendix \ref{proof-bound1}.
\end{proof}

  
\begin{thm}\label{minimumy1}
For the  linearized ADMM in Algorithm \ref{algorithm1}, under Assumption \ref{asp1},  if we choose the parameters $L_x$, $L_{y}$, and $\beta$ satisfying \eqref{bb}, then
the sequence $\{({\bf x}^{k},{\bf y}^{k},{\bm\gamma}^{k})\}$ satisfies
 \begin{align*}
\lim_{k\rightarrow \infty} \nabla_{{\bm\gamma}}L_{\beta}({\bf x}^{k},{\bf y}^{k},{\bm\gamma}^{k})=\lim_{k\rightarrow \infty}{\bf A}{\bf x}^{k}+{\bf B}{\bf y}^{k}={\bf 0},
\end{align*} 
 \begin{align*}
\lim_{k\rightarrow \infty} \nabla_{{\bf y}}L_{\beta}({\bf x}^{k},{\bf y}^{k},{\bm\gamma}^{k})={\bf 0},
\end{align*} 
and that there exits
\begin{equation}	
\bar{{\bf d}}^{k}\in\partial_{{\bf x}}L_{\beta}({\bf x}^{k},{\bf y}^{k},{\bm\gamma}^{k}) \mbox{ such that }  \lim_{k\rightarrow \infty}\bar{{\bf d}}^{k}={\bf 0}.
\end{equation}	
\end{thm}
  
\begin{proof}
The proof is postponed to Appendix \ref{proof-minimumy1}.
\end{proof}
  
Theorem \ref{minimumy1} illustrates that the sequence $\{({\bf x}^{k},{\bf y}^{k})\}$ will converge to the feasible set and the derivative of the Lagrangian function with respective to primal variables will converge to zero.  In other words, the limit points of $\{({\bf x}^{k},{\bf y}^{k})\}$, if exist, should be saddle points of $L_\beta$, alternatively KKT points to the original linearly constrained problem.

  
\begin{cor}\label{obj_converge}
For the  linearized ADMM in Algorithm \ref{algorithm1},  under Assumption \ref{asp1} together with function $g({\bf x},{\bf y})$ degenerating to $g({\bf x})$, if we choose the parameters $L_x$, $L_{y}$, and $\beta$ satisfying \eqref{bb}, then the sequence \{$g({\bf x}^k)+f({\bf x}^{k})+h({\bf y}^{k})$\} is convergent.
\end{cor}
  
\begin{proof}
The proof is postponed to Appendix \ref{proof-obj}.
\end{proof}

\begin{algorithm}
\caption{Multi-block parallel linearized ADMM algorithm}
\label{algorithm2}
\begin{algorithmic}
\STATE{ Initialize ${\bf x}^{0}, {\bf y}^{0}, {\bm\gamma}^{0}$.}\
\WHILE{$\max\{\|{\bf x}^{k}-{\bf x}^{k-1}\|,\|{\bf y}^{k}-{\bf y}^{k-1}\|,\|{\bm\gamma}^{k}-{\bm\gamma}^{k-1}\|\}>\varepsilon $}
\FOR{$i=1,\ldots,K$ \textbf{in parallel}}
\STATE{${\bf x}_{i}^{k+1}=\argmin_{{\bf x}_{i}}\bar{f}^{k}_{i}({\bf x}_{i})$}
\ENDFOR
\STATE{${\bf y}^{k+1}=\argmin_{{\bf y}} \bar{h}^k({\bf y})$}
\STATE{${\bm\gamma}^{k+1}={\bm\gamma}^{k}+\beta({\bf A}{\bf x}^{k+1}+{\bf B}{\bf y}^{k+1})$}
\STATE{$k = k+1$}
\ENDWHILE
\RETURN $({\bf x}^{k}, {\bf y}^{k}, {\bm\gamma}^{k})$
\end{algorithmic}
\end{algorithm}

\subsection{Multi-block parallel linearized ADMM}
\label{sec:multi-b-P-L-ADMM}

In this part, we focus the multi-block optimization problem \eqref{multi-block},
which can be seen as a special case of problem \eqref{prob:main},
where $f({\bf x})$ is further assumed to be separable across the blocks ${\bf x}_i$ for $i = 1,\ldots,K$.
We apply Algorithm \ref{algorithm1} to problem \eqref{multi-block}
and arrive at a multi-block linearized ADMM, which can update blocks of variables
in parallel even when they are coupled in the Lagrangian function.

To be specific,
because of  the linearization we use in ${\bf x}$-updating step,
the blocks ${\bf x}_1,\ldots,{\bf x}_K$ are decoupled in $\bar{f}^k({\bf x})$, so they can be optimized in parallel. In this case, we have $\bar{f}^k({\bf x}) = \sum_{i=1}^{K} \bar{f}^{k}_{i}({\bf x}_{i})$ where
 \begin{align}
 \bar{f}^{k}_{i}({\bf x}_{i})=f_{i}({\bf x}_{i})+\langle{\bm\gamma}^{k},{\bf A}_{i}{\bf x}_{i}\rangle+\frac{L_{x}}{2}\Vert {\bf x}_{i}-{\bf x}_{i}^{k}\Vert^{2} 
 +\big\langle {\bf x}_{i}-{\bf x}_{i}^{k},\nabla_{{\bf x}_i} g({\bf x}^{k},{\bf y}^{k})+\beta {\bf A}_i^{\rm T}({\bf A}{\bf x}^{k}+{\bf B}{\bf y}^{k})\big\rangle\label{eq-mblock-x-update0}
 \end{align} 
Utilizing the auxiliary functions \eqref{eq-twoblock-y-update} and \eqref{eq-mblock-x-update0}, the update rules are listed in Algorithm \ref{algorithm2}.
Similar to Algorithm \ref{algorithm1}, the updating rules for ${\bf x}$ and ${\bf y}$ in Algorithm \ref{algorithm2} can be simplified into the following form
	 \begin{align*}
	{\bf x}_i^{k+1} =& {\bf prox}_{f_i/L_x}\bigg\{ {\bf x}_i^{k}- \frac{1}{L_x}\big[{\bf A}_i^{\rm T} {\bm\gamma}^{k}+\nabla_{{\bf x}_i} g({\bf x}^{k},{\bf y}^{k})+
	\beta {\bf A}_i^{\rm T}({\bf A}{\bf x}^{k}+{\bf B}{\bf y}^{k})\big]\bigg\};\\
	{\bf y}^{k+1}=&\left(L_{y}+\beta {\bf B}^{\rm T} {\bf B}\right)^{-1}\big(L_{y}{\bf y}^{k}-\nabla_{{\bf y}} g({\bf x}^{k+1},{\bf y}^{k})
	-\nabla h({\bf y}^{k})-{\bf B}^{\rm T}{\bm\gamma}^{k}-\beta {\bf B}^{\rm T}{\bf A}{\bf x}^{k+1}\big).
	\end{align*} 
Because Algorithm \ref{algorithm2} can be seen as a special case of Algorithm \ref{algorithm1},
by replacing $f({\bf x})$ with $\sum_{i=1}^{K}f_{i}({\bf x}_{i})$
the theoretical convergence analyses for Algorithm \ref{algorithm1} can be directly applied to
Algorithm \ref{algorithm2}, so its convergence assumptions and results remain the same.

\section{Discussion}
\label{sec:discussion}

In this section, we give some discussion on our algorithms and their possible applications.

\subsection{Proximal term}
There are two main reasons why we use the proximal term in our algorithms. Firstly, in the proof of Lemma \ref{x-update} and \ref{y-update}, we will show that the descent of the Lagrangian function from updating primal variables is guaranteed due to the proximal term, so that we can put almost no restriction on $f$ ($f_i$). Secondly, the linearization skill used in our algorithm is a trade-off between the cost of solving subproblems and the accuracy of the solution to the subproblems, so it inevitably brings in \emph{inexactness}. Intuitively, the proximal term controls this \emph{inexactness} to be not too large so that our algorithm can converge.

\subsection{Linearization versus parallel computation}\label{lvsp}
As mentioned above, we actually get an inexact solution to the subproblems by solving the linearized subproblems, so intuitively the linearization skill would slow the convergence speed of the algorithm. It is indeed the case in our algorithm, but the linearization also decouples the variables coupled in the Lagrangian function, which reduces the time cost of a single iteration due to parallel computation. As a result, the time cost of the algorithm is determined by the balance between the deceleration from linearization and the acceleration from parallel computation. In  Section \ref{sec:exp}, we will empirically demonstrate that the acceleration can overwhelm the deceleration. Therefore, our algorithm can enjoy higher time efficiency in comparison with other nonconvex ADMM algorithms without linearization.

\subsection{Application}

In this part we present that the following general classes
of problems can meet the requirements in Assumption \ref{asp1}.
Consequently, our theorems guarantee the convergence of the algorithms,
if the problem belongs to one of the following commonly encountered classes.

\subsubsection{Sparsity relate topics}
Assume that $l({\bf x})$ is a loss function satisfying the following conditions.
\begin{itemize}
\item \textbf{Lipschitz Differentiability:} $l$ is differentiable, and there exits constant $L$ such that $\Vert \nabla l({\bf x}_1)-\nabla l({\bf x}_2)\Vert \le L\Vert {\bf x}_1 -{\bf x}_2\Vert$ for any ${\bf x}_1$, ${\bf x}_2$.
\item \textbf{Coercivity:} $l({\bf x})$ tends to infinity as $\Vert {\bf x} \Vert$ tends to infinity.
\end{itemize}
Then the following general sparsity related problem can be solved by our algorithm with convergence guarantee
\begin{equation}\label{dis_sparse}
\begin{array}{ll}
\mbox{minimize}&  \lambda\sum_{i=1}^{N} F( x_i)+l( {\bf y}- {\bf b}),\\
\mbox{subject to} & {\bf A}{\bf x}-{\bf y} = {\bf 0},
\end{array}
\end{equation}
where $F(\cdot)$ is some sparsity inducing function. For example, $F$ can be the $\ell_p$-norm ($0\le q \le 1$) or other nonconvex sparsity measures.
It is easy to verify that the above problem satisfies Assumption \ref{asp1}.

\subsubsection{Indicator function of compact manifold}

The indicator function of a compact manifold $\mathcal{M}$ is defined as follows
 \begin{align*}
\tau({\bf x})=
\begin{cases}
+\infty &  {\bf x}\notin \mathcal{M},\\
0 &  {\bf x}\in\mathcal{M}.
\end{cases}
\end{align*}

\begin{rem}\label{example}
Consider the following general form of problem,
where $\mathcal{M}$ is a finite subset of $\mathbb{Z}$, and $f$ is lower bounded over $\mathcal{M}$.
\begin{equation}\label{example_manifold}
\mbox{minimize} \,\, f({\bf y}) \quad
\mbox{subject to} \,\, {\bf y}\in \mathcal{M}.
\end{equation}
This problem is called \textbf{integer programming}
which is widely used in network design \cite{merchant1995assignment},
smart grid \cite{li2005price}, statistic learning \cite{glover1986future}, and other fields \cite{pallottino2002conflict}.
Problem \eqref{example_manifold} can be converted to the following
\begin{equation}\label{example_converted}
\begin{array}{ll}
\mbox{minimize} & \tau ({\bf x})+\big(f({\bf x})-h({\bf x})\big)+h({\bf y}) \\
\mbox{subject to} & {\bf x}={\bf y},
\end{array}
\end{equation}
where $\tau ({\bf x})$ is the indicator function of $\mathcal{M}$,
and function $h$ can be any nonzero Lipschitz function.
It can be verified that problem \eqref{example_converted} satisfies Assumption \ref{asp1}, if $h$ is Lipchitz differentiable.
In practice function $h$ can be appropriately chosen for solving the subproblems.
\end{rem}

\section{Numerical Experiment}\label{sec:exp}

In this section, we solve a nonconvex regularized LASSO by Algorithm \ref{algorithm2}
and two other reference ADMM algorithms,
in order to show the convergence behavior of our method and its advantage in run time
brought by parallel computation.

In sparsity related fields, many works have hinted that nonconvex penalties
can induce better sparsity than the convex ones (see, e.g., \cite{gasso2009recovering,chartrand2007exact,saab2010sparse} etc).
Our problem of interest is an improvement over LASSO,
where the traditional $l_1$-norm is replaced by a more effective nonconvex sparsity measure \cite{chen}.
The optimization problem is as the following
\begin{equation}	\label{exp1}
\mbox{minimize} \quad \lambda\sum_{i=1}^{N} F( x_i)+\Vert {\bf A}{\bf x}- {\bf b}\Vert^2,
\end{equation}	
where ${\bf x} \in \mathbb{R}^{N}$ is the variable,
and  ${\bf A}\in \mathbb{R}^{M\times N}$ and ${\bf b}\in\mathbb{R}^M$ are given.
The function $F$ is defined as
\begin{equation*}
F(t)=
\begin{cases}
   |t|-\eta t^2, &|t|\le \frac{1}{2\eta};\\
   \frac{1}{4\eta}, &|t| > \frac{1}{2\eta},
\end{cases}
\end{equation*}
where $\eta>0$ is a parameter and $F(t)$ is nonconvex and nonsmooth.

By introducing ${\bf y} = {\bf A}{\bf x}$, problem \eqref{exp1} is rewritten as
\begin{equation}\label{exp11}
\begin{array}{ll}
\mbox{minimize}&  \lambda\sum_{i=1}^{N} F( x_i)+\Vert {\bf y}- {\bf b}\Vert^2,\\
\mbox{subject to} & {\bf A}{\bf x}-{\bf y} = {\bf 0},
\end{array}
\end{equation}
where ${\bf x}$ and ${\bf y}$ are variables.
To the best of our knowledge, among the existing nonconvex ADMM algorithms, only the Algorithm $1$ in \cite{Wang} (referred as Ref1 here), the Algorithm $3$ in \cite{jiang2016structured} (referred as Ref2 here), and our algorithm can provide theoretical guarantee for the convergence of this problem.
We will compare the efficiency of these algorithms.

In the experiment, we set $N=1024$, $M=256$, $\lambda = 0.1$, and $\eta = 0.1$. Matrix ${\bf A}$ is a Gaussian random matrix and vector ${\bf b}$ is a Gaussian random vector.
In order to simplify the procedure of choosing parameters, matrix  ${\bf A}$ is normalized by a scalar,
so that the largest eigenvalue of ${\bf A}{\bf A}^{\rm T}$ is $1$.

For our algorithm, the parameters are set according to Theorem \ref{convergence1} as  $\beta = 12$, $L_x = 37$, and $L_y=8$,
and we implement the parallel computation by matrix multiplication in MATLAB.
For the reference algorithms, according to Lemma 7 and Lemma 9 in \cite{Wang} the parameter $\beta$
in Ref1 should be no less than $100$,
so we set it to be $100$, considering that the larger the $\beta$ is, the slower the convergence becomes.
Similarly, according to Theorem 3.18 in \cite{jiang2016structured},  we choose its parameters as $L=2$ and $\beta = 36$ in Ref2.
The stopping criterion of all these methods are set as
\begin{equation}	\label{epsilon}
 \max \big\{\|{\bf x}^{k}-{\bf x}^{k-1}\|,\|{\bf y}^{k}-{\bf y}^{k-1}\|,\|{\bf A} {\bf x}^{k}-{\bf y}^{k}\|\big\} < \epsilon.
\end{equation}	

\captionsetup[table]{name=Table}
\begin{table}[tbp]
\centering  
\caption{Comparison of average CPU running time with parameters chosen by theorems.}
\label{exptab1}
\begin{tabular}{lcc}  
\hline
Algorithm  &$\epsilon=10^{-3}$  &$\epsilon=10^{-4}$ \\ \hline  
Ref1 &$>1000$s &$>1000$s \\         
Ref2 &$162.76$s &$205.21$s \\        
Our Algorithm &${\bf 45.98}s$ &${\bf 60.91}s$ \\ \hline
\end{tabular}
\end{table}

\captionsetup[table]{name=Tabel}
\begin{table}[tbp]
\centering  
\caption{Comparison of average CPU running time with best parameters.}
\label{exptab2}
\begin{tabular}{lccc}  
\hline
Algorithm  &$\epsilon=10^{-3}$  &$\epsilon=10^{-4}$ &$\epsilon=10^{-7}$\\ \hline  
Ref1 &$230.23s$ &$257.84s$ &$ 439.3s$ \\         
Ref2 &$75.92$s &$83.31s$ &$109.57s$ \\        
Our Algorithm &${\bf 12.56}s$ &${\bf 15.23}s$ &${\bf 20.76}s$\\ \hline
\end{tabular}
\end{table}

We perform $1000$ independent trials on MATLAB 2016a with a 3.4 GHz Intel i7 processor,
and the ${\bf A}$ and ${\bf b}$ in each trial is generated randomly.
The average CPU running time is shown in Table \ref{exptab1}. We can see our algorithm enjoys higher time efficiency in
comparison with the other two algorithms.
In fact, the number of iterations of our method is around two times the numbers of iterations of the reference methods,
while their computing time for every iteration is around $7$ times of ours.  This corresponds with the analysis in section \ref{lvsp}.

Considering that the bounds on the parameters are not the tightest in our paper and the two references \cite{Wang,jiang2016structured},
the parameters chosen in the above experiment may not be the best for the three algorithms.
Therefore, we scan the parameters to find the best ones for every algorithm. For our algorithm, the best parameters found are $L_x = 1$, $L_y = 1$, and $\beta = 0.5$. For Ref1, the best parameter is $\beta=9.5$, and for Ref2 the best parameters are $L_y = 2$ and $\beta = 5.5$. We perform $1000$ independent trials with the best parameters again, and the average CPU running time is shown in Table \ref{exptab2}.
We can see that our algorithm still enjoys higher time efficiency in comparison with the other two algorithms.

Define the maximum variable gap as follows
 \begin{equation}	
 \max \big\{\|{\bf x}^{k}-{\bf x}^{k-1}\|,\|{\bf y}^{k}-{\bf y}^{k-1}\|,\|{\bf A} {\bf x}^{k}-{\bf y}^{k}\|\big\}.
 \end{equation}	
The curves of the maximum variable gap during the iterations in one random trial are plotted in Figure \ref{exp1fig0} which displays that our algorithm converges with the fastest speed.
Considering that the objective function in problem \eqref{exp1} is nonconvex and it may have more than one saddle point,
it is interesting to see where the value of objective function converges to,
so we plot its convergence curve in one random trail in Figure \ref{exp1fig1},
where the parameters are set as the same as the ones in the first experiment.
We can see that Ref2 and our algorithm converges to the same saddle point,
while Ref1 converges to another saddle point with a higher objective value.
We repeat the trial for $1000$ times with both the theoretically chosen parameters and the best parameters
and always obverse the same phenomenon.

\begin{figure}[t]
\center
\includegraphics[width=0.5\textwidth]{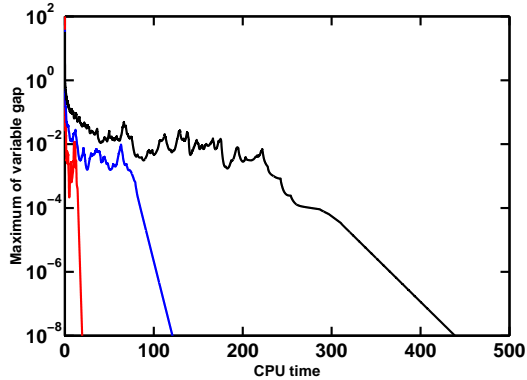}
\caption{Convergence curves of the maximum variable gap.
The red, blue, and black lines are our algorithm, Ref2, and Ref1, respectively. }
\label{exp1fig0}
\end{figure}

\begin{figure}[t]
\center
\includegraphics[width=0.5\textwidth]{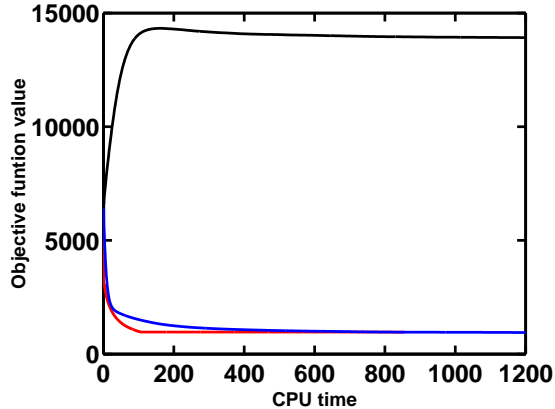}
\caption{Convergence curve of the objective function value.
The red, blue, and black lines are our algorithm, Ref2, and Ref1, respectively. }
\label{exp1fig1}
\end{figure}

\section{Conclusion}
\label{sec:conclusion}
In this work we study linearized ADMM algorithms for nonconvex optimization problems
with nonconvex nonsmooth objective function.
We propose a two-block linearized ADMM algorithm that introduces linearization for both
 the differentiable part in the objective and the augmented term,
and provide theoretical convergence analysis under Assumption \ref{asp1}.
Then we extend it to a multi-block parallel ADMM algorithm which can update coupled variables
in parallel and render subproblems easier to solve,
and the convergence analysis is still applicable.
By arguing that Assumption \ref{asp1} is not only plausible,
but also relatively broad compared with other recent works on ADMM for nonconvex optimization,
we show that the algorithms and their convergence analyses are general enough
to work for many interesting problems such as sparse recovery and integer programming.

\section{Appendix}

In this section, all notations ${\bf x}^k$, ${\bf y}^{k}$, and ${\bm\gamma}^{k}$
refer to the ones in Algorithm \ref{algorithm1} and $L_\beta$ refers to the augmented Lagrangian function defined in \eqref{Lagrangian}.
  
\begin{lemma}\label{subgradient-exist}
Suppose we have a differentiable function $f_1$, a possibly nondifferentiable function $f_2$, and a point ${\bf x}$.
If there exists ${\bf d}_2\in \partial f_2({\bf x})$, then we have
 $${\bf d} = {\bf d}_2-\nabla f_1({\bf x})\in \partial \left(f_2({\bf x})-f_1({\bf x})\right).$$
\end{lemma}
\begin{proof}
Firstly, by the definition of regular subgradient, we have
\begin{equation}	\label{aaa0}
f_2({\bf y}) \geq f_2({\bf x}) + \innerprod{{\bf d}_2}{{\bf y}-{\bf x}}+o(\Vert {\bf y}-{\bf x}\Vert).
\end{equation}	
Secondly, because function $f_1$ is differentiable, we have
\begin{equation}	\label{aaa1}
-f_1({\bf y}) = -f_1({\bf x}) - \innerprod{\nabla f_1({\bf x})}{{\bf y}-{\bf x}}+o(\Vert {\bf y}-{\bf x}\Vert).
\end{equation}	
Adding \eqref{aaa1} to \eqref{aaa0}, we get
 \begin{align*}
f_2({\bf y})-f_1({\bf y})
\geq f_2({\bf x})-f_1({\bf x})  + \innerprod{{\bf d}_2-\nabla f_1({\bf x})}{{\bf y}-{\bf x}}+o(\Vert {\bf y}-{\bf x}\Vert),
\end{align*} 
which together with the definition of regular subgradient leads to the conclusion.
\end{proof}

\begin{lemma}\label{Lipschitz}
	If $h({\bf y})$ is $L_{h}$-Lipschitz differentiable, then
	\begin{equation}
	h( {\bf y}_2)-h( {\bf y}_1)\ge \nabla h({\bf s})\cdot({\bf y}_2-{\bf y}_1)-\frac{L_h}{2}  \Vert{\bf y}_2-{\bf y}_1\Vert^2,
	\end{equation}
	where ${\bf s}$ denotes ${\bf y}_1$ or ${\bf y}_2$.
\end{lemma}
  

\begin{proof}
	 \begin{align*}
&	h( {\bf y}_2)-h( {\bf y}_1)\\
=&\int_{0}^{1}\nabla h(t {\bf y}_2 + (1-t) {\bf y}_1)\cdot ({\bf y}_2-{\bf y}_1){\rm d}t\\
	=& \int_{0}^{1} \nabla h({\bf s})\cdot ({\bf y}_2-{\bf y}_1){\rm d}t
	+ \int_{0}^{1} \big(\nabla h(t {\bf y}_2 + (1-t) {\bf y}_1)- \nabla h({\bf s})\big)\cdot({\bf y}_2-{\bf y}_1){\rm d}t,
	\end{align*} 
	where $\nabla h(\cdot)$ defines the gradient of $h(\cdot)$. If we take ${\bf s}={\bf y}_1$, then by inequality
	$$\Vert\nabla h(t {\bf y}_2 + (1-t) {\bf y}_1)- \nabla h({\bf y}_1)\Vert\le L_h \Vert t({\bf y}_2-{\bf y}_1)\Vert $$
	we have
	 \begin{align*}
	&\int_{0}^{1} \nabla h({\bf y}_1)\cdot({\bf y}_2-{\bf y}_1){\rm d}t
	+ \int_{0}^{1} \big(\nabla h(t {\bf y}_2 + (1-t) {\bf y}_1)- \nabla h({\bf y}_1)\big)\cdot({\bf y}_2-{\bf y}_1){\rm d}t\\
	\ge &  \nabla h({\bf y}_1)\cdot({\bf y}_2-{\bf y}_1)-\int_{0}^{1} L_h t \Vert{\bf y}_2-{\bf y}_1\Vert^2{\rm d}t\\
	=&  \nabla h({\bf y}_1)\cdot({\bf y}_2-{\bf y}_1)-\frac{L_h}{2}  \Vert{\bf y}_2-{\bf y}_1\Vert^2.
	\end{align*} 
	Therefore, we get
	$$  h( {\bf y}_2)-h( {\bf y}_1)\ge \nabla h({\bf y}_1)\cdot({\bf y}_2-{\bf y}_1)-\frac{L_h}{2}  \Vert{\bf y}_2-{\bf y}_1\Vert^2.$$
	Similarly, if we take ${\bf s}={\bf y}_2$, we can get
	$$  h( {\bf y}_2)-h( {\bf y}_1)\ge \nabla h({\bf y}_2)\cdot({\bf y}_2-{\bf y}_1)-\frac{L_h}{2}  \Vert{\bf y}_2-{\bf y}_1\Vert^2.$$
\end{proof}
  
\begin{lemma}\label{gammaB}
	Under Assumption \ref{asp1}, for any $l>k$, we have
	\begin{equation*}
	\Vert {\bm\gamma}^{l}-{\bm\gamma}^{k}\Vert^{2}\leq \frac{1}{\lambda_{{\bf B}^{\rm T}{\bf B}}}\Vert {\bf B}^{\rm T}({\bm\gamma}^{l}-{\bm\gamma}^{k})\Vert^{2},
	\end{equation*}
	where $\lambda_{{\bf B}^{\rm T}{\bf B}}$ is the smallest eigenvalue of ${\bf B}^{\rm T}{\bf B}$.
\end{lemma}
  
\begin{proof}
	By the ${\bm\gamma}$-updating rule and the assumption ${\bf Im}({\bf A})\subset {\bf Im}({\bf B})$,
	for two integers $l>k$, we have
	\begin{equation*}
	{\bm\gamma}^{l}-{\bm\gamma}^{k}=\sum_{i=k+1}^{l}\beta({\bf A}{\bf x}^{i}+{\bf B}{\bf y}^{i})\in {\bf Im}({\bf B}).
	\end{equation*}
	Because ${\bf B}\in \mathbb{R}^{n\times q}$ has full column rank, there exists ${\bf R}\in \mathbb{R}^{q\times q}, {\bf Q}\in \mathbb{R}^{q\times n}$ such that
	${\bf R}$ is invertible, ${\bf Q}{\bf Q}^{\rm T}={\bf I}_{n\times n}$, and ${\bf B}^{\rm T}={\bf R}{\bf Q}$.
	Noticing that ${\bf Im}({\bf B})={\bf Im}({\bf Q}^{\rm T})$, we get ${\bm\gamma}^{l}-{\bm\gamma}^{k}\in{\bf Im}({\bf Q}^{\rm T})$.
	Thus, $\Vert {\bm\gamma}^{l}-{\bm\gamma}^{k}\Vert^{2}=\Vert {\bf Q}({\bm\gamma}^{l}-{\bm\gamma}^{k})\Vert^{2}$.
	Consequently, we have
	 \begin{align*}
	\Vert {\bf B}^{\rm T}({\bm\gamma}^{l}-{\bm\gamma}^{k})\Vert^{2} &= \Vert {\bf R}{\bf Q}({\bm\gamma}^{l}-{\bm\gamma}^{k})\Vert^{2}\\
	&\geq \lambda_{{\bf R}^{\rm T}{\bf R}}\Vert {\bf Q}({\bm\gamma}^{l}-{\bm\gamma}^{k})\Vert^{2}\\
	&=\lambda_{{\bf R}^{\rm T}{\bf R}}\Vert {\bm\gamma}^{l}-{\bm\gamma}^{k}\Vert^{2},
	\end{align*} 
	where $\lambda_{{\bf R}^{\rm T}{\bf R}}$ denotes the minimum eigenvalue of ${\bf R}^{\rm T}{\bf R}$.
	
	By the definition of ${\bf R}$ and ${\bf Q}$, we have $\lambda_{{\bf B}^{\rm T}{\bf B}}=\lambda_{{\bf R}{\bf R}^{\rm T}}$. Together with the common conclusion in linear algebra $\lambda_{{\bf R}^{\rm T}{\bf R}}=\lambda_{{\bf R}{\bf R}^{\rm T}}$, we get $\lambda_{{\bf R}^{\rm T}{\bf R}}=\lambda_{{\bf B}^{\rm T}{\bf B}}$, which completes the proof.
\end{proof}
  

\begin{lemma}\label{gammaUpdate}
	Under Assumption \ref{asp1}, the following equality holds for ${\bm\gamma}^{k+1}$, ${\bf y}^{k}$, and ${\bf y}^{k+1}$
	\begin{equation*}
	{\bf B}^{\rm T}{\bm\gamma}^{k+1}=-\nabla_{\bf y} g({\bf x}^{k+1},{\bf y}^{k})-\nabla h({\bf y}^{k})-L_{y} ({\bf y}^{k+1}-{\bf y}^{k}).
	\end{equation*}
\end{lemma}
  
\begin{proof}
	By calculating the derivative of $\bar{h}^k({\bf y})$ defined in \eqref{eq-twoblock-y-update}, we have
	 \begin{align*}
	\nabla\bar{h}^k({\bf y})=\nabla_{\bf y} g({\bf x}^{k+1},{\bf y}^{k})+\nabla h({\bf y}^{k})+L_{y}({\bf y}-{\bf y}^{k})
	+{\bf B}^{\rm T}{\bm\gamma}^{k}+\beta {\bf B}^{\rm T}({\bf A}{\bf x}^{k+1}+{\bf B}{\bf y}).
	\end{align*} 
	Plug ${\bf y}={\bf y}^{k+1}$ into it,
	and by the ${\bf y}$-updating rule we have
	\begin{equation}	\label{Eq7}
	 {\bf B}^{\rm T}{\bm\gamma}^{k}+\beta {\bf B}^{\rm T}({\bf A}{\bf x}^{k+1}+{\bf B}{\bf y}^{k+1})
	=-\nabla_{\bf y} g({\bf x}^{k+1},{\bf y}^{k})-\nabla h({\bf y}^{k})-L_{y}({\bf y}^{k+1}-{\bf y}^{k}).
	\end{equation}
	Besides, by the ${\bm\gamma}$-updating rule, we have
	\begin{equation}\label{Eq8}
	{\bf B}^{\rm T}{\bm\gamma}^{k+1} = {\bf B}^{\rm T}{\bm\gamma}^{k}+\beta {\bf B}^{\rm T}({\bf A}{\bf x}^{k+1}+{\bf B}{\bf y}^{k+1}).
	\end{equation}
	By replacing the RHS of \eqref{Eq8} with \eqref{Eq7}, we get
	\begin{equation*}
	{\bf B}^{\rm T}{\bm\gamma}^{k+1}=-\nabla_{\bf y} g({\bf x}^{k+1},{\bf y}^{k})-\nabla h({\bf y}^{k})-L_{y} ({\bf y}^{k+1}-{\bf y}^{k}).
	\end{equation*}
\end{proof}

Lemma \ref{gammaUpdate} provides a way to express ${\bm\gamma}^{k+1}$ using ${\bf y}^{k}$ and ${\bf y}^{k+1}$,
which is a technique widely used in the convergence proof for nonconvex ADMM algorithms \cite{Wang,Hong}.

Now we are ready to prove Theorem \ref{convergence1}.
We first give bounds on the descent or ascent of the Lagrangian function \eqref{L1}
after every update
by using the quadratic form of the primal residual.
Specifically, in the following,
Lemma \ref{x-update} presents that the descent of $L_{\beta}$ is lower bounded after the ${\bf x}$-updating step,
Lemma \ref{y-update} shows that the descent of $L_{\beta}$ is lower bounded after the ${\bf y}$-updating step,
and Lemma \ref{gamma-update} demonstrates that the ascent of $L_{\beta}$ is upper bounded after the ${\bm\gamma}$-updating step.
  
\begin{lemma}\label{x-update}
	Under Assumption \ref{asp1}, the following inequality holds for the update of ${\bf x}$
	\begin{equation*}
	L_{\beta}({\bf x}^{k},{\bf y}^{k},{\bm\gamma}^{k})-L_{\beta}({\bf x}^{k+1},{\bf y}^{k},{\bm\gamma}^{k}) \ge C_0 \Vert {\bf x}^{k+1}-{\bf x}^{k}\Vert^{2},
	\end{equation*}
	where $C_0 = \frac{L_{x}-L_{g}-\beta L_{\bf A}}{2}$ and $L_{\bf A}$ denotes the largest singular value of ${\bf A}^{\rm T}{\bf A}$.
\end{lemma}
  
\begin{proof}
	By ${\bf x}$-updating rule in Algorithm \ref{algorithm1}, we have
	\begin{equation}\label{a1}
	\bar{f}^k({\bf x}^{k})\ge \bar{f}^k({\bf x}^{k+1}).
	\end{equation}
	Plugging the definition of $\bar{f}^k$ in \eqref{eq-twoblock-x-update} into \eqref{a1}, we get
	 \begin{align}\label{a2}
	\nonumber &f({\bf x}^{k})-f({\bf x}^{k+1})+ \langle {\bf x}^{k}-{\bf x}^{k+1},\beta {\bf A}^{\rm T}({\bf A}{\bf x}^{k}+{\bf B}{\bf y}^{k})+{\bf A}^{\rm T}{\bm\gamma}^{k}\rangle \\
	& \ge         \langle {\bf x}^{k+1}-{\bf x}^{k},  \nabla_{\bf x} g({\bf x}^{k},{\bf y}^{k}) \rangle    + \frac{L_{x}}{2}\Vert {\bf x}^{k+1}-{\bf x}^{k}\Vert^{2}.
	\end{align} 
	Then we have
	 \begin{align*}
	\nonumber &L_{\beta}({\bf x}^{k},{\bf y}^{k},{\bm\gamma}^{k})- L_{\beta}({\bf x}^{k+1},{\bf y}^{k},{\bm\gamma}^{k})\\
	\nonumber  =&f({\bf x}^{k})+g({\bf x}^{k},{\bf y}^{k})-f({\bf x}^{k+1})-g({\bf x}^{k+1},{\bf y}^{k})\\
	\nonumber &+\langle{\bf \gamma}^k,{\bf A}{\bf x}^{k}-{\bf A}{\bf x}^{k+1} \rangle+\frac{\beta}{2} \Vert {\bf A}{\bf x}^{k}+{\bf B}{\bf y}^{k} \Vert ^{2}-\frac{\beta}{2} \Vert {\bf A}{\bf x}^{k+1}+{\bf B}{\bf y}^{k} \Vert ^{2}\\
	\nonumber =&f({\bf x}^{k})+g({\bf x}^{k},{\bf y}^{k})-f({\bf x}^{k+1})-g({\bf x}^{k+1},{\bf y}^{k})+ \\
	\nonumber &\langle {\bf x}^{k}-{\bf x}^{k+1},{\bf A}^{\rm T}{\bm\gamma}^{k}\rangle+ \langle {\bf x}^{k}-{\bf x}^{k+1},\beta {\bf A}^{\rm T}({\bf A}{\bf x}^{k}+{\bf B}{\bf y}^{k})\rangle-\frac{\beta}{2} \Vert {\bf A}({\bf x}^{k+1}-{\bf x}^{k})\Vert ^{2}\\
	\nonumber \ge&g({\bf x}^{k},{\bf y}^{k})-g({\bf x}^{k+1},{\bf y}^{k})+\langle {\bf x}^{k+1}-{\bf x}^{k},  \nabla_{\bf x} g({\bf x}^{k},{\bf y}^{k}) \rangle\\
&+\frac{L_{x}}{2}\Vert {\bf x}^{k+1}-{\bf x}^{k}\Vert^{2}-\frac{\beta}{2} \Vert {\bf A}({\bf x}^{k+1}-{\bf x}^{k})\Vert ^{2}\\
	\ge&  \frac{L_{x}-L_{g}-\beta L_{\bf A}}{2}\Vert {\bf x}^{k+1}-{\bf x}^{k}\Vert^{2},
	\end{align*} 
	where the last inequality is from Lemma \ref{Lipschitz} and  $L_{\bf A}$ denotes the largest singular value of ${\bf A}^{\rm T}{\bf A}$.
\end{proof}

\begin{lemma}\label{y-update}
	Under Assumption \ref{asp1}, the following inequality holds for the update of ${\bf y}$
	$$
	L_{\beta}({\bf x}^{k+1},{\bf y}^{k},{\bm\gamma}^{k})- L_{\beta}({\bf x}^{k+1},{\bf y}^{k+1},{\bm\gamma}^{k})
	\geq C_1\Vert {\bf y}^{k}-{\bf y}^{k+1}\Vert^{2},
	$$
	\ where $C_1 = \frac{2L_y - L_\omega}{2}$ and $L_\omega = L_g + L_h$.
\end{lemma}  
\begin{proof}
	
	According to that $\bar{h}^k({\bf y})$ is $L_{y}$-convex, by Proposition 4.8 in \cite{Weakly_convex} we have
	 \begin{align*}
	\bar{h}^k({\bf y}^{k})
	 \geq \bar{h}^k({\bf y}^{k+1})+\left\langle {\bf y}^{k}-{\bf y}^{k+1},\nabla\bar{h}({\bf y}^{k+1})\right\rangle
	+\frac{L_{y}}{2}\Vert {\bf y}^{k}-{\bf y}^{k+1}\Vert^{2}.
	\end{align*} 
	According to the updating rule of $\bf y$, i.e., $\nabla\bar{h}^k({\bf y}^{k+1})={\bf 0}$,
	the above inequality is reshaped to
	\begin{equation}\label{Eq3}
	\bar{h}^k({\bf y}^{k})\geq \bar{h}^k({\bf y}^{k+1})+\frac{L_{y}}{2}\Vert {\bf y}^{k}-{\bf y}^{k+1}\Vert^{2}.
	\end{equation}
	Denote $w^{k}({\bf y}) = g({\bf x}^{k},{\bf y})+ h({\bf y})$ and  recall that $g({\bf x},{\bf y})$ and  $h({\bf y})$ are $L_g$ and $L_h$ Lipschitz-differentiable, respectively. We get that $w^k ({\bf y})$ is $L_w $  Lipschitz-differentiable, where $L_w = L_g+L_h$.  Then by Lemma \ref{Lipschitz} we have
	\begin{equation}\label{Eq4}
	 w^{k+1}({\bf y}^{k})
	\geq w^{k+1}({\bf y}^{k+1})
	+\left\langle {\bf y}^{k}-{\bf y}^{k+1},\nabla w^{k+1}({\bf y}^{k})\right\rangle
	-\frac{L_{w}}{2}\Vert {\bf y}^{k}-{\bf y}^{k+1}\Vert^{2}.
	\end{equation}
	Now we consider the descent of $L_\beta$ in $\bf y$-updating step.
	 \begin{align}
	& L_{\beta}({\bf x}^{k+1},{\bf y}^{k},{\bm\gamma}^{k})-L_{\beta}({\bf x}^{k+1},{\bf y}^{k+1},{\bm\gamma}^{k}) \nonumber\\
	= & \ w^{k+1}({\bf y}^{k})-w^{k+1}({\bf y}^{k+1})+\left\langle {\bm\gamma}^{k},{\bf B}({\bf y}^{k}-{\bf y}^{k+1}) \right\rangle\\&
	+\frac{\beta}{2}\Vert {\bf A}{\bf x}^{k+1}+{\bf B}{\bf y}^{k}\Vert^{2}-\frac{\beta}{2}\Vert {\bf A}{\bf x}^{k+1}+{\bf B}{\bf y}^{k+1}\Vert^{2}. \label{eq-y-update-1}
	\end{align} 
	By plugging \eqref{Eq4} into \eqref{eq-y-update-1}, we have
	 \begin{align}
	{\rm RHS \, of \, } \eqref{eq-y-update-1} \geq& \left\langle {\bf y}^{k}-{\bf y}^{k+1},\nabla w^{k+1}({\bf y}^{k})\right\rangle-\frac{L_{w}}{2}\Vert {\bf y}^{k}-{\bf y}^{k+1}\Vert^{2}\nonumber\\
	&+\langle {\bm\gamma}^{k},{\bf B}({\bf y}^{k}-{\bf y}^{k+1}) \rangle +\frac{\beta}{2}\Vert {\bf A}{\bf x}^{k+1}+{\bf B}{\bf y}^{k}\Vert^{2}
	-\frac{\beta}{2}\Vert {\bf A}{\bf x}^{k+1}+{\bf B}{\bf y}^{k+1}\Vert^{2}.\label{eq-y-update-A1}
	\end{align} 
	By the definition of $\bar{h}^k({\bf y})$ in \eqref{eq-twoblock-y-update}, we further derive
	\begin{equation}
	{\rm RHS \, of \, } \eqref{eq-y-update-A1}  = \bar{h}^k({\bf y}^{k})- \bar{h}^k({\bf y}^{k+1})+\frac{L_{y}-L_{w}}{2}\Vert {\bf y}^{k}-{\bf y}^{k+1}\Vert^{2}. \label{eq-y-update-2}
	\end{equation}
	By inserting \eqref{Eq3} into \eqref{eq-y-update-2}, we finally reach
	$$
	L_{\beta}({\bf x}^{k+1},{\bf y}^{k},{\bm\gamma}^{k})-L_{\beta}({\bf x}^{k+1},{\bf y}^{k+1},{\bm\gamma}^{k}) \geq C_1\Vert {\bf y}^{k}-{\bf y}^{k+1}\Vert^{2},
	$$
	where
	\begin{equation*}
	C_1 := \frac{2L_{y}-L_{w}}{2}\label{eq-define-C0}.
	\end{equation*}
\end{proof}

\begin{lemma}\label{gamma-update}
	Under Assumption \ref{asp1}, the following inequality holds for the update of ${\bm\gamma}$
	 \begin{align}
	\nonumber &L_{\beta}({\bf x}^{k+1},{\bf y}^{k+1},{\bm\gamma}^{k+1})-L_{\beta}({\bf x}^{k+1},{\bf y}^{k+1},{\bm\gamma}^{k})\\
	=&  \frac{1}{\beta}\Vert {\bm\gamma}^{k+1}-{\bm\gamma}^{k}\Vert^{2} \nonumber \\
	\leq & C_2\Vert {\bf x}^{k+1}-{\bf x}^{k}\Vert^{2} +C_3\Vert {\bf y}^{k+1}-{\bf y}^{k} \Vert^{2}+C_4\Vert {\bf y}^{k}-{\bf y}^{k-1}\Vert^{2},\label{eq00}
	\end{align} 
	where $L_w = L_g+L_h$, $C_2= \frac{3L^{2}_{w}}{\beta \lambda_{{\bf B}^{\rm T}{\bf B}}}$, $C_3=\frac{3L_{y}^2 }{\beta \lambda_{{\bf B}^{\rm T}{\bf B}}}$ and $C_4=\frac{3(L^{2}_{w}+L_{y}^2 )}{\beta \lambda_{{\bf B}^{\rm T}{\bf B}}}$.
\end{lemma}  

\begin{proof}
	By definition, the ascent of $L_\beta$ after the $(k+1)$th iteration of ${\bf \gamma}$ is
	\begin{equation}	
	L_{\beta}({\bf x}^{k+1},{\bf y}^{k+1},{\bm\gamma}^{k+1})-L_{\beta}({\bf x}^{k+1},{\bf y}^{k+1},{\bm\gamma}^{k})
	= \left\langle{\bm\gamma}^{k+1}-{\bm\gamma}^{k},{\bf A}{\bf x}^{k+1}+{\bf B}{\bf y}^{k+1}\right\rangle.\label{eq-proof-gamma-update-1}
	\end{equation}	
	By inserting the ${\bm\gamma}$-updating rule in \eqref{eq-proof-gamma-update-1} and applying Lemma \ref{gammaB}, we have
	 \begin{align}
	 L_{\beta}({\bf x}^{k+1},{\bf y}^{k+1},{\bm\gamma}^{k+1})-L_{\beta}({\bf x}^{k+1},{\bf y}^{k+1},{\bm\gamma}^{k})
	=  &\frac{1}{\beta}\Vert {\bm\gamma}^{k+1}-{\bm\gamma}^{k}\Vert^{2} \nonumber\\ \leq& \frac{1}{\beta \lambda_{{\bf B}^{\rm T}{\bf B}}}\Vert {\bf B}^{\rm T}({\bm\gamma}^{k+1}-{\bm\gamma}^{k})\Vert^{2},\label{eq-proof-gamma-update-2}
	\end{align} 
	where $\lambda_{{\bf B}^{\rm T}{\bf B}}$ denotes the smallest singular value of ${\bf B}^{\rm T}{\bf B}$.\\
	By Lemma \ref{gammaUpdate} and AM-GM Inequality we have
	 \begin{align}
	& \Vert {\bf B}^{\rm T}({\bm\gamma}^{k+1}-{\bm\gamma}^{k})\Vert^{2}\\
 =& \big\Vert\nabla w^{k+1}({\bf y}^{k})+L_{y}({\bf y}^{k+1}-{\bf y}^{k})-\nabla w^{k}({\bf y}^{k-1})
	-L_{y}({\bf y}^{k}-{\bf y}^{k-1})\big\Vert^{2} \\
	 \leq& 3\bigg(\Vert\nabla w^{k+1}({\bf y}^{k})-\nabla w^{k}({\bf y}^{k-1})\Vert^{2}
	+ L^2_{y}\Vert {\bf y}^{k+1}-{\bf y}^{k}\Vert^{2}
	+L_{y}^2 \Vert {\bf y}^{k}-{\bf y}^{k-1}\Vert^{2}\bigg)\label{eq-proof-gamma-update-3},
	\end{align} 
	where $w^{k}({\bf y}) = g({\bf x}^{k},{\bf y})+ h({\bf y})$ has been defined in the proof of Lemma \ref{y-update}.\\
	Because $g({\bf x},{\bf y})+ h({\bf y})$ is $L_w$ Lipschitz differentiable, we have
	 \begin{align*}
	&\Vert\nabla w^{k+1}({\bf y}^{k})-\nabla w^{k}({\bf y}^{k-1})\Vert^{2}\\
	=&\Vert\nabla_{\bf y} g({\bf x}^{k+1},{\bf y}^{k})+\nabla h({\bf y}^{k})-\nabla_{\bf y} g({\bf x}^{k},{\bf y}^{k-1})-\nabla h({\bf y}^{k-1})\Vert^{2}\\
	\le & L^{2}_{w}\left( \Vert {\bf x}^{k+1}-{\bf x}^{k}\Vert^{2}+\Vert {\bf y}^{k}-{\bf y}^{k-1}\Vert^{2}\right),
	\end{align*} 
	and together with \eqref{eq-proof-gamma-update-2} and \eqref{eq-proof-gamma-update-3} we have
	 \begin{align*}
	&L_{\beta}({\bf x}^{k+1},{\bf y}^{k+1},{\bm\gamma}^{k+1})-L_{\beta}({\bf x}^{k+1},{\bf y}^{k+1},{\bm\gamma}^{k})\\
	\le& C_2\Vert {\bf x}^{k+1}-{\bf x}^{k}\Vert^{2} + C_3\Vert {\bf y}^{k+1}-{\bf y}^{k} \Vert^{2}+C_4\Vert {\bf y}^{k}-{\bf y}^{k-1}\Vert^{2},
	\end{align*} 
	where
	 \begin{align}
	C_{2}  &:= \frac{3L^{2}_{w}}{\beta \lambda_{{\bf B}^{\rm T}{\bf B}}},\label{eqC2}\\
	C_{3} &:=\frac{3L_{y}^2 }{\beta \lambda_{{\bf B}^{\rm T}{\bf B}}},\label{eq_define_C1}\\
	C_{4} &:=\frac{3(L^{2}_{w}+L_{y}^2 )}{\beta \lambda_{{\bf B}^{\rm T}{\bf B}}}.\label{eq_define_C2}
	\end{align} 
\end{proof}

Then we design a sequence $\{m_{k}\}^{+\infty}_{k=1}$ by
\begin{equation}\label{eq-m_{k}}
m_{k} = L_{\beta}({\bf x}^{k},{\bf y}^{k},{\bm\gamma}^{k})+C_{m}\Vert {\bf y}^{k}-{\bf y}^{k-1}\Vert^{2},
\end{equation}
where $C_{m}$ is set according to \eqref{bb} in Theorem $1$. We will first prove the convergence of $\{m_{k}\}^{+\infty}_{k=1}$ and
then prove the convergence of $\{L_{\beta}({\bf x}^{k},{\bf y}^{k},{\bm\gamma}^{k})\}$.
  
\begin{lemma}\label{m_{k}}
	For the  linearized ADMM in Algorithm \ref{algorithm1},  under Assumption \ref{asp1}, if we choose the parameters $L_x$, $L_{y}$ and $\beta$ satisfying \eqref{bb}, then
	the sequence $\{m_{k}\}$ defined in \eqref{eq-m_{k}} is convergent.
\end{lemma}  
\begin{proof}
	$\quad$

	1. Monotonicity of $\{m_k\}$
	
	By using Lemma \ref{x-update}, Lemma \ref{y-update}, and Lemma \ref{gamma-update}, we have
	 \begin{align}
	&L_\beta({\bf x}^k,{\bf y}^k,{\bm\gamma}^k) - L_\beta({\bf x}^{k+1},{\bf y}^{k+1},{\bm\gamma}^{k+1})\nonumber\\
	\ge & \ L_\beta({\bf x}^{k+1},{\bf y}^k,{\bm\gamma}^k) - L_\beta({\bf x}^{k+1},{\bf y}^{k+1},{\bm\gamma}^{k+1})+C_0\Vert {\bf x}^{k+1}-{\bf x}^{k}\Vert^{2}\nonumber\\
	\ge & \ L_\beta({\bf x}^{k+1},{\bf y}^{k+1},{\bm\gamma}^k)
	- L_\beta({\bf x}^{k+1},{\bf y}^{k+1},{\bm\gamma}^{k+1})
	+ C_1\Vert{\bf y}^{k+1} - {\bf y}^{k}\Vert^2+ C_0\Vert {\bf x}^{k+1}-{\bf x}^{k}\Vert^{2}\nonumber\\
	\ge & \ (C_1-C_3)\| {\bf y}^{k+1}-{\bf y}^k\|^2 - C_4\Vert {\bf y}^{k}-{\bf y}^{k-1}\Vert^{2}
	+ (C_0-C_2)\Vert {\bf x}^{k+1}-{\bf x}^{k}\Vert^{2}.\label{Eq11-0}
	\end{align} 
	By combining \eqref{Eq11-0} with the definition of $m_k$, we have
	 \begin{align}\label{Eq11}
   m_{k}-m_{k+1}
	\geq&(C_1-C_3-C_{m})\Vert {\bf y}^{k+1}-{\bf y}^{k}\Vert^{2}+(C_{m}-C_{4})\Vert {\bf y}^{k}-{\bf y}^{k-1}\Vert^{2}\nonumber
	\\&+  (C_0-C_2)\Vert {\bf x}^{k+1}-{\bf x}^{k}\Vert^{2}.
	\end{align} 
	Recall the definition of $C_0$, $C_1$, $C_3$, $C_4$ and the parameters $L_{x}$, $L_{y}$, $C_{m}$, $\beta$ we choose in \eqref{bb}, we get
	 \begin{align}
	C_0-C_2&=\frac{1}{2}(L_{x}-L_{g}-\beta L_{\bf A}-\frac{6L^{2}_{w}}{\beta \lambda_{{\bf B}^{\rm T}{\bf B}}}) \ge \frac{1}{2},\label{eq000}\\
	C_1-C_3-C_{m} &= \frac{2L_{y}-L_{w}}{2}-\frac{3L^{2}_{y}}{\beta \lambda_{{\bf B}^{\rm T}{\bf B}}}-C_{m} \ge \frac12 ,\label{eq003}\\
	C_{m}-C_4 &= C_{m}-\frac{3(L^{2}_{w}+L^{2}_{y})}{\beta \lambda_{{\bf B}^{\rm T}{\bf B}}}>0.\label{eq004}
	\end{align} 
	Therefore, $\{m_{k}\}$ is monotonically decreasing.
	
	2. Lower bound of $\{m_k\}$
	
	Next we will argue that $\{m_{k}\}$ is also lower bounded.
	By the assumption ${\bf Im}({\bf A})\subset {\bf Im}({\bf B})$, there exists ${\bf y}'_{k}$ such that ${\bf B}{\bf y}'_{k}=-{\bf A}{\bf x}^{k}$, so we have
	 \begin{align}
	m_{k}=&g({\bf x}^{k},{\bf y}^k)+f({\bf x}^{k})+h({\bf y}^{k})+\langle{\bm\gamma}^{k},{\bf B}({\bf y}^{k}-{\bf y}'_{k})\rangle \\&+\frac{\beta}{2}\Vert {\bf B}({\bf y}^{k}-{\bf y}'_{k})\Vert^{2}+C_{m}\Vert {\bf y}^{k}-{\bf y}^{k-1}\Vert^{2}.\label{eq-proof-m_{k}-1}
	\end{align} 
	By applying Lemma \ref{gammaUpdate} to the third item in the RHS of \eqref{eq-proof-m_{k}-1}, we have
	 \begin{align}
&	 \langle{\bm\gamma}^{k},{\bf B}({\bf y}^{k}-{\bf y}'_{k})\rangle\\
	= & \ \langle{\bf B}^{\rm T}{\bm\gamma}^{k},{\bf y}^{k}-{\bf y}'_{k}\rangle \nonumber\\
	= & \ \langle-\nabla w^k ({\bf y}^{k-1})-L_{y}({\bf y}^{k}-{\bf y}^{k-1}),{\bf y}^{k}-{\bf y}'_{k}\rangle\nonumber\\
	=& \ \big\langle\nabla w^k({\bf y}^{k})-\nabla w^k({\bf y}^{k-1}) -L_{y}({\bf y}^{k}-{\bf y}^{k-1}),{\bf y}^{k}-{\bf y}'_{k}\big\rangle
	- \langle\nabla w^k({\bf y}^{k}),{\bf y}^{k}-{\bf y}'_{k}\rangle. \label{eq-proof-m_{k}-2}
	\end{align} 
	By AM-GM Inequality, we bound the first item in the RHS of \eqref{eq-proof-m_{k}-2}
	 \begin{align}\label{eq-proof-m_{k}-3}
	\nonumber &\langle\nabla w^k({\bf y}^{k})-\nabla w^k({\bf y}^{k-1}) -L_{y}({\bf y}^{k}-{\bf y}^{k-1}),{\bf y}^{k}-{\bf y}'_{k}\rangle\\
	\nonumber =&\langle\nabla w^k({\bf y}^{k})-\nabla w^k({\bf y}^{k-1}) ,{\bf y}^{k}-{\bf y}'_{k}\rangle
	 -  L_{y}\langle{\bf y}^{k}-{\bf y}^{k-1},{\bf y}^{k}-{\bf y}'_{k}\rangle \\
	 \geq & -\frac12 \bigg(\Vert\nabla w^k({\bf y}^{k})-\nabla w^k({\bf y}^{k-1})\Vert^{2}+\Vert {\bf y}^{k}-{\bf y}'_{k}\Vert^{2} \bigg) \nonumber\\&-\frac{L_y}2\bigg(\Vert {\bf y}^{k}-{\bf y}^{k-1}\Vert^{2} +\Vert {\bf y}^{k}-{\bf y}'_{k}\Vert^{2}\bigg)\nonumber \\
	\geq& -\frac12 \left((L_{w}^{2}+L_{y})\Vert {\bf y}^{k}-{\bf y}^{k-1}\Vert^{2} + (L_{y}+1)\Vert {\bf y}^{k}-{\bf y}'_{k}\Vert^{2}\right),
	\end{align} 
	where the last inequality is from
	the Lipschitz differentiability of $w^k({\bf y})$.
	
	Considering that $\bf B$ has full rank and $\|{\bf B}{\bf z}\|^2 \ge \lambda_{{\bf B}^{\rm T}{\bf B}}\|{\bf z}\|^2$, for all
	$\bf z$,
	the fourth item in the RHS of \eqref{eq-proof-m_{k}-1} can be bounded by
	\begin{equation}\label{eq-proof-m_{k}-4}
	\Vert {\bf B}({\bf y}^{k}-{\bf y}'_{k})\Vert^{2}\geq \lambda_{{\bf B}^{\rm T}{\bf B}}\Vert {\bf y}^{k}-{\bf y}'_{k}\Vert^{2},
	\end{equation}
	By plugging \eqref{eq-proof-m_{k}-2}, \eqref{eq-proof-m_{k}-3}, and \eqref{eq-proof-m_{k}-4} into \eqref{eq-proof-m_{k}-1}, we get
	$$
	m_{k}\geq Q_1^k + Q_2^k,\label{eq40}
	$$
	where
	 \begin{align*}
	Q_1^k := &g({\bf x}^{k},{\bf y}^k)+f({\bf x}^{k}) + h({\bf y}^{k}) -\langle\nabla w^k({\bf y}^{k}),{\bf y}^{k}-{\bf y}'_{k}\rangle
	 \\&+\frac{1}{2}\left(\beta\lambda_{{\bf B}^{\rm T}{\bf B}}-L_{y}-1\right)\Vert {\bf y}^{k}-{\bf y}'_{k}\Vert^{2},\\
	Q_2^k :=& \left(C_{m}-\frac{L_{y}}{2}-\frac{L_{w}^{2}}{2}\right)\Vert {\bf y}^{k}-{\bf y}^{k-1}\Vert^{2}.
	\end{align*} 
	If both $Q_1^k$ and $Q_2^k$ are lower bounded, the proof will be completed.
	Let us first check $Q_2^k$.
	Recall the $C_{m}$ and $L_{y}$ we choose in \eqref{bb}, we get
	\begin{equation}\label{eq002}
	Q_2^k = \left(C_{m}-\frac{L_{y}}{2}-\frac{L_{w}^{2}}{2}\right)\Vert {\bf y}^{k}-{\bf y}^{k-1}\Vert^{2} = 0.
	\end{equation}
	For $Q_1^k$, recall the $\beta$ and $L_{y}$ we choose in \eqref{bb} and  we get
	\begin{equation}\label{eq001}
	\beta\lambda_{{\bf B}^{\rm T}{\bf B}} \ge L_{w}+L_{y}+2,
	\end{equation}
	then by Lemma \ref{Lipschitz} we have
	 \begin{align}
	\nonumber Q_1^k \ge &g({\bf x}^{k},{\bf y}^k)+f({\bf x}^{k}) + h({\bf y}^{k}) -\langle\nabla w^k({\bf y}^{k}),{\bf y}^{k}-{\bf y}'_{k}\rangle
	+\frac{L_{w}}{2}\Vert {\bf y}^{k}-{\bf y}'_{k}\Vert^{2}+\frac{1}{2}\Vert {\bf y}^{k}-{\bf y}'_{k}\Vert^{2}\nonumber\\
	\ge &g({\bf x}^{k},{\bf y}'_k)+f({\bf x}^{k})+h({\bf y}'_{k})+\frac{1}{2}\Vert {\bf y}^{k}-{\bf y}'_{k}\Vert^{2},\nonumber
	\end{align} 
	where $g({\bf x}^{k},{\bf y}'_k)+f({\bf x}^{k})+h({\bf y}'_{k})$ is lower bounded, because $({\bf x}^{k},{\bf y}'_k)$ belongs to the feasible set. Therefore, $\{m_{k}\}$ is lower bounded.
	Together with its monotonic decrease, we get $\{m_{k}\}$ is convergent.
\end{proof}

\subsection{Proof of Theorem \ref{convergence1}}
\label{proof-convergence1}

Recall in Lemma \ref{m_{k}}, we first prove that $\{m_k\}$ is monotonically decreasing by
 \begin{align*}
m_{k}-m_{k+1}
\geq & (C_1-C_3-C_{m}) \Vert {\bf y}^{k+1} - {\bf y}^{k}\Vert^{2}
+ (C_{m}-C_4)\Vert {\bf y}^{k}-{\bf y}^{k-1}\Vert^{2}
\\&+   (C_0-C_2)\Vert {\bf x}^{k+1}-{\bf x}^{k}\Vert^{2},
\end{align*} 
and then prove that $\{m_k\}$ is lower bounded by
\begin{equation}	\label{t1}
m_{k}\geq g({\bf x}^{k},{\bf y}'_k)+f({\bf x}^{k})+h({\bf y}'_{k})+\frac{1}{2}\Vert {\bf y}^{k}-{\bf y}'_{k}\Vert^{2},
\end{equation}	
where ${\bf y}'_{k}$ is defined by ${\bf B}{\bf y}'_{k}=-{\bf A}{\bf x}^{k}$.
Notice that ${\bf y}'_{k}$ always exists because of the assumption ${\bf Im}({\bf A})\subset {\bf Im}({\bf B})$.

By the convergence of $\{m_{k}\}$, $\Vert {\bf x}^{k+1}-{\bf x}^{k}\Vert$ and $\Vert {\bf y}^{k+1}-{\bf y}^{k}\Vert$ converges to zero.
By the definition of $\{m_{k}\}$ and its convergence, we readily get the convergence of
$L_{\beta}({\bf x}^{k},{\bf y}^{k},{\bm\gamma}^{k})$.
According to Lemma \ref{gamma-update}, $\Vert {\bm\gamma}^{k+1}-{\bm\gamma}^{k}\Vert$ converges to zero as well.

\subsection{Proof of Corollary \ref{bound1}}
\label{proof-bound1}

Recall \eqref{t1} in the proof of Lemma \ref{m_{k}}.
Because $g({\bf x},{\bf y})+f({\bf x})+h({\bf y})$ is coercive over the feasible set with respect to ${\bf y}$,
if $\big\{{\bf y}'_{k}\big\}$ diverges, then the RHS of \eqref{t1} diverges to positive infinity,
which contradicts with the convergence of $\{m_{k}\}$.

Because of the term $\frac{1}{2}\Vert {\bf y}^{k}-{\bf y}'_{k}\Vert^{2}$ on the RHS of \eqref{t1}, the boundedness of $\{{\bf y}^{k}\}$ can be derived from the boundedness of $\{{\bf y}'_{k}\}$.

In order to prove that $\{{\bm\gamma}^{k}\}$ is bounded, we only need to prove $\{{\bm\gamma}^{k}-{\bm\gamma}^{0}\}$ is bounded. By Lemma \ref{gammaB}, it is equivalent to the boundedness of $\{{\bf B}^{\rm T}({\bm\gamma}^{k}-{\bm\gamma}^{0})\}$ and further equivalent to the boundedness of $\{{\bf B}^{\rm T}{\bm\gamma}^{k}\}$. When  function $g({\bf x},{\bf y})$ degenerates to $g({\bf x})$, by Lemma \ref{gammaUpdate}, we get
$$
{\bf B}^{\rm T}{\bm\gamma}^{k+1}=-\nabla h({\bf y}^{k})-L_{y} ({\bf y}^{k+1}-{\bf y}^{k}),
$$
which implies that the boundedness of $\{{\bf B}^{\rm T}{\bm\gamma}^{k}\}$ can be deduced from the boundedness of $\{{\bf y}^{k}\}$.

\subsection{Proof of Theorem \ref{minimumy1}}
\label{proof-minimumy1}
$\quad$

1. Limit of $\nabla_{{\bm\gamma}}L_{\beta}$

When $k$ approaches infinity, we have
 \begin{align*}
\nabla_{{\bm\gamma}}L_{\beta}({\bf x}^{k+1},{\bf y}^{k+1},{\bm\gamma}^{k+1})&
={\bf A}{\bf x}^{k+1}+{\bf B}{\bf y}^{k+1}\\
&=\frac{1}{\beta}({\bm\gamma}^{k+1}-{\bm\gamma}^{k})\rightarrow{\bf 0}.
\end{align*} 

2. Limit of $\nabla_{\bm{y}}L_{\beta}$

By  Theorem \ref{convergence1} and Lemma \ref{gammaUpdate}, when $k$ approaches infinity, we have
 \begin{align*}
\nabla_{{\bf y}}L_{\beta}({\bf x}^{k},{\bf y}^{k},{\bm\gamma}^{k})
=& \nabla_{\bf y} g({\bf x}^{k},{\bf y}^{k})+\nabla h({\bf y}^{k})+{\bf B}^{\rm T}{\bm\gamma}^{k}+\beta {\bf B}^{\rm T}({\bf A}{\bf x}^{k}+{\bf B}{\bf y}^{k})\\
\rightarrow &\nabla_{\bf y} g({\bf x}^{k},{\bf y}^{k-1})+\nabla h({\bf y}^{k-1})+{\bf B}^{\rm T}{\bm\gamma}^{k}+{\bf B}^{\rm T}({\bm\gamma}^{k}-{\bm\gamma}^{k-1})\\
=&-L_{y}({\bf y}^{k}-{\bf y}^{k-1})+{\bf B}^{\rm T}({\bm\gamma}^{k}-{\bm\gamma}^{k-1})\rightarrow {\bf 0}.
\end{align*} 

3. Limit of $\partial_{\bm{x}}L_{\beta}$

By ${\bf x}$-updating rule, ${\bf x}^{k+1}$ is the minimum point of $\bar{f}^k({\bf x})$, which implies ${\bf 0}\in\partial\bar{f}^k({\bf x}^{k+1})$. Therefore, by the definition of $\bar{f}^k$ in \eqref{eq-twoblock-x-update} and Lemma \ref{subgradient-exist}, there exists
${\bf d}^{k+1}\in\partial f({\bf x}^{k+1})$ such that
\begin{equation}	\label{eqa1}
 \nabla_{\bf x} g({\bf x}^{k},{\bf y}^{k})+{\bf d}^{k+1}+{\bf A}^{\rm T}{\bm\gamma}^{k}+\beta {\bf A}^{\rm T}({\bf A}{\bf x}^k+{\bf B}{\bf y}^{k})
+L_{x}({\bf x}^{k+1}-{\bf x}^{k})={\bm 0}.
\end{equation}	
We further define
 \begin{align}\label{eqa2}
\nonumber \bar{{\bf d}}^{k+1}:=&\nabla_{\bf x} g({\bf x}^{k+1},{\bf y}^{k+1})+{\bf d}^{k+1}+{\bf A}^{\rm T}{\bm\gamma}^{k+1}
 +\beta {\bf A}^{\rm T}\big({\bf A}{\bf x}^{k+1}+{\bf B}{\bf y}^{k+1}\big),
\end{align} 
which, one may readily check, satisfies
$$
\bar{{\bf d}}^{k+1}\in\partial_{{\bf x}}L_{\beta}({\bf x}^{k+1},{\bf y}^{k+1},{\bm\gamma}^{k+1}).
$$
By Theorem \ref{convergence1}, we have that the primal residues $\Vert {\bf y}^{k+1}-{\bf y}^{k}\Vert$, $\Vert {\bf x}^{k}-{\bf x}^{k+1} \Vert$
and dual residue $\Vert {\bm\gamma}^{k+1}-{\bm\gamma}^{k}\Vert$ converge to zero as $k$ approaches infinity, therefore
 \begin{align*}
&\lim\limits_{k\rightarrow +\infty} \bar{{\bf d}}^{k+1}\\
=&\lim\limits_{k\rightarrow +\infty} \left[\nabla_{\bf x} g({\bf x}^{k+1},{\bf y}^{k+1})+{\bf d}^{k+1}+{\bf A}^{\rm T}{\bm\gamma}^{k+1}
+\beta {\bf A}^{\rm T}({\bf A}{\bf x}^{k+1}+{\bf B}{\bf y}^{k+1})\right]\\
=&\lim\limits_{k\rightarrow +\infty}\left[ \nabla_{\bf x} g({\bf x}^{k},{\bf y}^{k})+{\bf d}^{k+1}+{\bf A}^{\rm T}{\bm\gamma}^{k}
+\beta {\bf A}^{\rm T}({\bf A}{\bf x}^k+{\bf B}{\bf y}^{k})+L_{x}({\bf x}^{k+1}-{\bf x}^{k})\right]\\
=&{\bf 0},
\end{align*} 
where the last equality is from \eqref{eqa1}.

\subsection{Proof of Corollary \ref{obj_converge}}\label{proof-obj}
As $k$ tends to infinity, by Corollary \ref{bound1} and Theorem \ref{minimumy1}, we have that $\gamma^k$ is bounded and ${\bf A}{\bf x}^k+{\bf B}{\bf y}^k \rightarrow 0$.
Then we have that
 \begin{align*}
f({\bf x}^{k})+h({\bf y}^{k})
=& L_{\beta}({\bf x}^{k},{\bf y}^{k},{\bf \gamma}^{k})-\langle{\bf \gamma}^k,{\bf A}{\bf x}^k+{\bf B}{\bf y}^k \rangle
-\frac{\beta}{2} \Vert {\bf A}{\bf x}^k+{\bf B}{\bf y}^k \Vert ^{2}\\
\rightarrow& L_{\beta}({\bf x}^{k},{\bf y}^{k},{\bf \gamma}^{k})-0-0=L_{\beta}({\bf x}^{k},{\bf y}^{k},{\bf \gamma}^{k}).
\end{align*} 
Therefore, the value of objective function will converge, because $L_\beta$ will converge.

%

\bibliographystyle{unsrt}%
\bibliography{reference}
\end{document}